\documentclass[11pt,leqno]{amsart}
\usepackage[margin=1in]{geometry}
\usepackage{amsmath, amsthm, amssymb}
\usepackage{amsfonts}
\usepackage[colorlinks=true, pdfstartview=FitV, linkcolor=blue,citecolor=blue, urlcolor=blue]{hyperref}
\usepackage[T1]{fontenc}
\usepackage[utf8]{inputenc}
\usepackage[foot]{amsaddr}
\usepackage{parskip}

\theoremstyle{theorem}
\newtheorem{thm}{Theorem}[section]

\theoremstyle{corollary}

\theoremstyle{lemma}
\newtheorem{lemma}{Lemma}[section]

\theoremstyle{definition}
\newtheorem{definition}{Definition}[section]

\theoremstyle{proposition}
\newtheorem{proposition}{Proposition}[section]

\theoremstyle{proof}

\theoremstyle{boldremark}
\newtheorem{rem}{Remark}[section]

\title{Time-optimal control of a two-peakon collision }
\author{Tomasz Cie\'slak and Bidesh Das}
\address[A1,A2]{Institute of Mathematics, Polish Academy of Sciences, Warsaw, 00-656, Poland }
\email[A1,A2]{cieslak@impan.pl and dasbiiser014@gmail.com }
\date{}
\begin{document}

\begin{titlepage}
\begin{abstract}
\hspace{0.6cm}We study a time-optimal control problem of a two-peakon collision. First, we state the controllability. Next, we find the time-optimal strategy. This is done via the Hamilton-Jacobi-Bellman equation and the dynamic programming method. We find a tricky way to solve a corresponding Hamilton-Jacobi-Bellman equation. Our approach can be extended to more general control problems. Indeed, a special version of a maximum principle for time-optimal controls of particular Hamiltonian systems is proposed.
\end{abstract}
\maketitle
\section{Introduction}
\vspace{0.5cm}
Multipeakons are particular solutions of the Camassa-Holm equation considered already in \cite{CH}. We say that $u:[0,T]\times\mathbb{R}\to\mathbb{R}$ is a solution to the Camassa-Holm equation if $u$ satisfies
\begin{eqnarray}\label{1}	
u_t-u_{xxt}+3uu_x-2u_xu_{xx}-uu_{xxx}=0.
\end{eqnarray}
Such an equation models waves in shallow water \cite{CH}. An $n$-peakon solution of equation(\ref{1}) is given by the formula
\begin{eqnarray}\label{2}
	u(x,t)=\sum\limits_{i=1}^np_i(t)e^{-|x-q_i(t)|},
\end{eqnarray}
where $t\geq0$, $x\in\mathbb{R}$ and $p(t)=\left(p_1(t),\ldots,p_n(t)\right),
q(t)=\left(q_1(t),\ldots,q_n(t)\right)$. In order to make sure that $(\ref{2})$ solves $(\ref{1})$, the Hamiltonian system
\begin{eqnarray}\label{3}
	\dot{q_i}=\frac{\partial H}{\partial p_i}, \hspace{2cm} \dot{p_i}=-\frac{\partial H}{\partial q_i},\ \ \ i=1,\ldots,n,
\end{eqnarray}
with Hamiltonian \ $H(p,q):=\frac{1}{2}\sum\limits_{i,j=1}^np_ip_je^{-|q_i-q_j|}$, has to be satisfied.

We say that a multipeakon $u(x,t)$, given by formula $(\ref{2})$, collides at time $t^*> 0$ if $q_i(t^*) = q_j(t^*)$ for some distinct $i$ and $j$. If it is the case then $t^*$ is referred to as a collision time for $u$ and $q^*= q(t^*)$ is referred to as the corresponding collision point.

The studies of the dynamics of multipeakons were started already in \cite{CH}. The necessary and sufficient condition of collision for multipeakons and explicit formula of collision time have been obtained in \cite{BSS}. The geometric studies of two-peakon have been started in \cite{CGKM}, \cite{Wk}, \cite{TW}. Specially, in \cite{Wk} the author presented a bi-Hamiltonian formulation of the system explicitly and wrote down formulae for the associated first integrals.

In this paper, we study equation controllability as well as time-optimal controls of collision of two-peakons. We solve the controllabilty issue of control problem $(\ref{6})$. First, in section 2, the control problem related to the collision of a two-peakon is introduced. Also, the main theorem is stated. Its proof is obtained in the next two sections. In section 3 we reduce the initially four dimensional Hamiltonian system, related to the controllability of collisions of two-peakon, to a plane problem using a special change of variables. Then the controllability is proven in a straightforward way. Plane time-optimal problems are better understood, at least since the important contribution of \cite{HS} and the deep characterization in \cite{BP}. Still, our problem is quite singular, we shall explain on this issue further, and so the switching function approach of Sussmann is not applicable in a straightforward way. Instead of adjusting it to our situation, we rather use the dynamic programming method. In section 4, a  Hamilton-Jacobi-Bellman equation for the shortest time value function is found. It turns out to be a nonlinear Hamilton-Jacobi equation with a completely non-standard boundary condition. Again, a standard theory of viscosity solution does not seem to be applicable without essential modifications. Still, we are able to solve the Hamilton-Jacobi equation in a tricky way. Next, a standard procedure allows us to identify a time-optimal control, which turns out to be of no-switching form. Our method seems quite new. It has some consequences for more general control problems of Hamiltonian form. In section five, as a consequence of our attitude, we give more general theorems on time-optimal controls for a special type of control problems. 

Finally, let us mention that we are not the first to deal with controllability issues related to the Camassa-Holm equation, $(\ref{1})$. The controllability of the full equation $(\ref{1})$ has been discussed in \cite{OG}. It is a problem of a different nature, however related to our activity.  

\section{Hamiltonian system for Camassa-Holm Equation with control parameters}

The present section is devoted to the introduction of a control problem as well as formulation of the main theorem. 

We aim to study a control problem where controls act exactly on peaks and throughs of the peakon. Let us introduce control parameters in $(\ref{1})$. Hamiltonian system, describing the evolution of a multipeakon solution to $(\ref{1})$, is already defined in $(\ref{3})$. We shall arrive at the description of the evolution of a two-peakeon, solving the Camassa-Holm equation with additional, very specific controls. We mention that the global result on exact controllability and asymptotic stabilization of the Camassa-Holm equation on the circle have already been obtained in \cite{OG}. \\
\\
Introducing controls $a_1,a_2\in[-\alpha,\alpha]$ acting on peak's heights only, $(\ref{1})$ reads
\begin{eqnarray}\label{6}
	u_t-u_{xxt}+3uu_x-2u_xu_{xx}-uu_{xxx}+a_1\delta(x-q_1)+a_2\delta(x-q_2)=0.
\end{eqnarray}
Now we find a system of equations satisfied by a two-peakon solving $(\ref{6})$. Our two-peakon solution is not differentiable in a classical sense. So, all derivatives occuring in the proof, are meaningful in a sense of distributions. Let us recall that $(sign(x))'=2\delta$, $(e^{-|x|})'=-sign(x)e^{-|x|}$ and $(e^{-|x|}sign(x))'=2\delta-e^{-|x|}$ in the sense of distributions.

Before stating main result, we recall an useful lemma on distributions.
\begin{lemma}\label{4}
Let us assume $f:\mathbb{R}\to\mathbb{R}$ is a Lipschitz continuous function. Then the dirac delta distribution $\delta_{x_0}\in \mathcal{D}'(\mathbb{R})$ satisfies the relation $f\delta_{x_0}'[\varphi]=-f\delta_{x_0}[\varphi']-f'\delta_{x_0}[\varphi]$ where $\varphi\in\mathcal{D}(\mathbb{R})$ and $x_0\in\mathbb{R}$ are arbitrary.
\end{lemma}
The following proposition is simillar to proving that a multipeakon $(\ref{2})$ is a solution to the Camassa-Holm equation $(\ref{1})$. Proving the following proposition, we shall follow the exposition in \cite{Wk}.
\begin{proposition}\label{5}
Let $a_1,a_2:[0,t]\to[-\alpha,\alpha], \ \alpha>0$, are control parameters and $u(x,t)$ is a two-peakon solution of (\ref{6}). Then $u$ satisfies the following Hamiltonian system
\begin{align}
		\dot{p_i}&=\sum\limits_{j=1}^2p_ip_je^{-|q_i-q_j|}sign(q_i-q_j)-\frac{a_i}{2}, \label{7}\\
		\dot{q_i}&=\sum\limits_{j=1}^2p_je^{-|q_i-q_j|}\label{10}.
\end{align}
In other words, $(\ref{7})$ and $(\ref{10})$ can be written down as a Hamiltonian with $H_a(p,q)=H_1+\frac{1}{2}\left(a_1q_1+a_2q_2\right)$, where $H_1=\frac{1}{2}\left(p_1^2+2p_1p_2e^{-|q_1-q_2|}+p_2^2\right)$.
\end{proposition}
\vspace{-0.6cm}
\begin{proof}
	Take a two-peakon $u(x,t)=\sum\limits_{i=1}^2p_ie^{-|x-q_i|}$. We compute the derivatives
\begin{align*}
	 u_t&=\sum\limits_{i=1}^2\dot{p_i}e^{-|x-q_i|}+\sum\limits_{i=1}^2p_i\dot{q_i}e^{-|x-q_i|}sign(x-q_i),\\u_x&=-\sum\limits_{i=1}^2p_ie^{-|x-q_i|}sign(x-q_i),\\u_{xx}&=\sum\limits_{i=1}^2p_ie^{-|x-q_i|}-2\sum\limits_{i=1}^2p_i\delta(x-q_i),\\u_{xxx}&=-\sum\limits_{i=1}^2p_ie^{-|x-q_i|}sign(x-q_i)-2\sum\limits_{i=1}^2p_i\delta'(x-q_i),\\u_{xxt}&=\sum\limits_{i=1}^2\dot{p_i}e^{-|x-q_i|}+\sum\limits_{i=1}^2p_i\dot{q_i}e^{-|x-q_i|}sign(x-q_i)-2\sum\limits_{i=1}^2\dot{p_i}\delta(x-q_i)+2\sum\limits_{i=1}^2p_i\dot{q_i}\delta'(x-q_i),
\end{align*}
where $\delta$ is the Dirac distribution and $\delta'$ is its derivative.\\
\\
Putting the above in (\ref{6}), we get
\begin{align*}
	 2\sum\limits_{i=1}^2\delta(x-q_i)\left(\dot{p_i}-2p_i\sum\limits_{j=1}^2p_je^{-|x-q_j|}sign(x-q_j)\right)\\-2\sum\limits_{i=1}^2\delta'(x-q_i)\left(p_i\dot{q_i}-p_i\sum\limits_{j=1}^2p_je^{-|x-q_j|}\right)\\+a_1\delta(x-q_1)+a_2\delta(x-q_2)=0.
\end{align*}
Evaluating this on any $\varphi\in \mathcal{D}(\Omega)$ and having in mind $f\delta'[\varphi]=-f'\delta[\varphi]-f\delta[\varphi']$, from \autoref{4} we get
\begin{align*}
	 2\sum\limits_{i=1}^2\delta(x-q_i)\left(\dot{p_i}-p_i\sum\limits_{j=1}^2p_je^{-|x-q_j|}sign(x-q_j)\right)[\varphi]\\+2\sum\limits_{i=1}^2\delta(x-q_i)p_i\left(\dot{q_i}-\sum\limits_{j=1}^2p_je^{-|x-q_j|}\right)[\varphi']\\+a_1\delta(x-q_1)[\varphi]+a_2\delta(x-q_2)[\varphi]=0,
\end{align*}
which is equivalent to
\begin{align*}
\delta(x-q_1)\left(\dot{p_1}-p_1\sum\limits_{j=1}^2p_je^{-|x-q_j|}sign(x-q_j)+a_1/2\right)[\varphi]\\+\delta(x-q_2)\left(\dot{p_2}-p_2\sum\limits_{j=1}^2p_je^{-|x-q_j|}sign(x-q_j)+a_2/2\right)[\varphi]\\
+\sum\limits_{i=1}^2\delta(x-q_i)p_i\left(\dot{q_i}-\sum\limits_{j=1}^2p_je^{-|x-q_j|}\right)[\varphi']=0.
\end{align*}
Since $\varphi$ is arbitrary, we obtain
\begin{align*}
\dot{p_i}&=\sum\limits_{j=1}^2p_ip_je^{-|q_i-q_j|}sign(q_i-q_j)-\frac{a_i}{2},\\
\dot{q_i}&=\sum\limits_{j=1}^2p_je^{-|q_i-q_j|},\ \ \mbox{for} \ \ i=1,2,
\end{align*}
in the sense of distributions. In other words, $(\ref{7})$ and $(\ref{10})$ are Hamilton's equations with Hamiltonian
\begin{align}\label{H}
	H_a(p,q)=\frac{1}{2}\left(p_1^2+2p_1p_2e^{-|q_1-q_2|}+p_2^2\right)+\frac{1}{2}\left(a_1q_1+a_2q_2\right).
\end{align}
\end{proof}
First, let us recall that $(\ref{3})$ is an integrable system (strictly speaking the
 system is Liouville integrable outside the singular set $q_1 = q_2$) and it possesses two independent first integrals: $H_0=p_1+p_2$ and $H_1=\frac{1}{2}\left(p_1^2+2p_1p_2e^{-|q_1-q_2|}+p_2^2\right)$. 
 
However, in our case $(\ref{7})-(\ref{10})$, the constants of motion are not preserved, due to the presence of time-dependent controls. Still, the Hamiltonian structure, with a Hamiltonian $H_a$ is preserved. 

All our effort in the present paper is related to a very particular situation. Namely, we consider only controls $a_1$ and $a_2$ satisfying $a_1=-a_2$. In particular, this way we make sure $(\ref{7})-(\ref{10})$ possesses a constant of motion, indeed $H^a=p_1+p_2$ is a first integral. We have the following result. 

\begin{rem}\label{2.1}
Let us consider the special case $a_2=-a_1$. Then $(\ref{H})$ reads
\begin{align*}
	H_a(p,q)=\frac{1}{2}\left(p_1^2+2p_1p_2e^{-|q_1-q_2|}+p_2^2\right)+\frac{a_1}{2}\left(q_1-q_2\right).
\end{align*}
Equivalently, Hamiltonian associated with $(\ref{7})-(\ref{10})$ is
\begin{align}\label{H_1}
	H_a(p_1,p_2,q_1,q_2)=\frac{1}{2}p^TE(q)p+\frac{a_1}{2}\left(q_1-q_2\right),
\end{align}
where $E(q)=\left[e^{-|q_i-q_j|}\right]_{i,j=1}^2$ and $p=\left(p_1,p_2\right)$, $q=\left(q_1,q_2\right)$.\\
\\
Now, in view of $(\ref{H_1})$, we see that the evolution $(\ref{7})-(\ref{10})$, with $a_1=-a_2$, is equivalent to
\begin{eqnarray}
	\dot{p_1} &=& p_1p_2e^{-|q_1-q_2|}-\frac{a_1}{2}\label{11},\\
	\dot{p_2} &=& -p_1p_2e^{-|q_1-q_2|}+\frac{a_1}{2}\label{12}, \\
	\dot{q_1} &=& p_1+p_2e^{-|q_1-q_2|}\label{13},\\
	\dot{q_2} &=& p_2+p_1e^{-|q_1-q_2|}\label{14}.
\end{eqnarray}
From $(\ref{11})$ and $(\ref{12})$, we notice that $p_1+p_2$=const.
\end{rem}
Now we are ready to state the main result of the paper. It concerns both, the controllability of a collision, as well as pointing the time-optimal control. Let us recall that a collision occurs when $q_1(t^*)=q_2(t^*)$ for a finite time $t^*$.
\begin{thm}\label{2.}
We assume that $a_1$ is a measurable function $a_1:[0,T]\rightarrow [-\alpha, \alpha]$. Let us consider solutions to $(\ref{11})-(\ref{14})$. Without loss of generality, assume $q_1(0)<q_2(0)$. Assume moreover that $\alpha\geq 1/2(p_1(0)+p_2(0))^2$. 

We say that a solution $(\ref{11})-(\ref{14})$ collides at time $t^*<\infty$ if $q_1(t^*)=q_2(t^*)$. Controllability of a collision holds for $\alpha\geq\frac{1}{2}\left(p_1(0)+p_2(0)\right)^2$. Moreover, the time-optimal control $a_1$ leading to the collision is constant and equal to $a_1=-\alpha$. 
\end{thm}
\section{Controllability of the problem}

In this section, we discuss a controllability of a collision of solution to $(\ref{11})-(\ref{14})$. In particular, we prove the first part of our main Theorem \ref{2.}. This is achieved thanks to the change of variables which reduces the complexity of the problem essentially. The above-mentioned change of varibales is also extremely helpful in our proof of the time-optimal control.
\\
Let us once more recall what is understood by a collision. A two-peakon $u(x,t)$, given by (\ref{2}), collides at time $t^*>0$ if $q_1(t^*)=q_2(t^*)$. If this happens then $t^*$ is referred to as a collision time and $q^*=q_1(t^*)=q_2(t^*)$ is referred to as the corresponding collision point.
In order to show controllability of the problem, it is sufficient to pick up such a control $a_1$, for which $q_1$ and $q_2$ satisfying (\ref{11})-(\ref{14}), satisfy $q_1-q_2\to 0$ for some fixed time $t^*<\infty$.

Problem $(\ref{11})-(\ref{14})$ is four dimensional. We introduce a particular change of variables which reduces $(\ref{11})-(\ref{14})$ to the plane problem. Let us define new coordinates (as in \cite{CGKM})
\begin{align}\label{C}
s_1:=\frac{q_1+q_2}{2}, \ s_2:=\frac{q_1-q_2}{2}.
\end{align}
Without loss of generality, consider the situation $q_1(0)<q_2(0)$. In new coordinates $(\ref{C})$, occurence of a collision in time $t^*$ means $s_2(t^*)=0$. 

Before stating the reduction lemma, let us speak a few words about its proof. Actually the proof could be easier. We could simply introduce the corresponding change of variables in $p$ and check that the corresponding equations are satisfied. Instead, let us provide a slightly more complicated proof, which on the other hand, follows our reasoning. Our change of variables was done at the Lagrangian level. Then, the proper covariables occurred in a natural way, when going back from Lagrangian to Hamiltonian coordinates.
  
We observe that the evolution $(\ref{11})-(\ref{14})$ is equivalent to the following evolution in new coordinates.
\begin{proposition}\label{8}
Let $q_1(0)<q_2(0)$, $(\ref{11})-(\ref{14})$ in coordinates $s_1, s_2$ transforms to
\begin{eqnarray} \dot{\tilde{p_1}}&=&0,\label{19}\\\dot{\tilde{p_2}}&=&\frac{1}{2}\left(\tilde{p_2}^2-\tilde{p_1}^2\right)e^{2s_2}-a_1,\label{20}\\\dot{s_1}&=&\frac{\tilde{p_1}}{2}\left(1+e^{2s_2}\right),\label{21}\\\dot{s_2}&=&\frac{\tilde{p_2}}{2}\left(1-e^{2s_2}\right)\label{22},
\end{eqnarray}
where $\tilde{p_1}$ and $\tilde{p_2}$ are as mentioned in \autoref{9}. In particular, $\tilde{p_1}$ is the constant of motion.
\end{proposition}
\begin{rem}\label{9}
Let us consider $\tilde{p_1}$, $\tilde{p_2}$ from \autoref{8}. Then, $\tilde{p_1}=p_1+p_2$ and $\tilde{p_2}=p_1-p_2$.
\end{rem}
\begin{proof}
In view of $(\ref{C})$, we obtain from $(\ref{21})$ and $(\ref{22})$ that
\begin{align}
\dot{q_1}+\dot{q_2}&=\tilde{p_1}\left(1+e^{(q_1-q_2)}\right),\label{q_1}\\\dot{q_1}-\dot{q_2}&=\tilde{p_2}\left(1-e^{(q_1-q_2)}\right)\label{q_2}.
\end{align}
Since $\dot{q}=E(q)p$, we have
\begin{align*}
\begin{bmatrix}
\dot{q_1}\\\dot{q_2}
\end{bmatrix}=\begin{bmatrix}
1&e^{(q_1-q_2)}\\e^{(q_1-q_2)}&1
\end{bmatrix}\begin{bmatrix}
p_1\\p_2
\end{bmatrix}.
\end{align*}
Owing to $(\ref{q_1}),(\ref{q_2})$, we arrive at
\begin{align*}
\left(1+e^{(q_1-q_2)}\right)(p_1+p_2)&=\left(1+e^{(q_1-q_2)}\right)\tilde{p_1},
\end{align*}
which implies
\begin{align*}
\tilde{p_1}=p_1+p_2.
\end{align*}
Moreover, similarly
\begin{align*}
	\left(1-e^{(q_1-q_2)}\right)(p_1-p_2)&=\left(1-e^{(q_1-q_2)}\right)\tilde{p_2},
\end{align*}
hence
\begin{align*}
\tilde{p_2}=p_1-p_2.	
\end{align*}
\end{proof}
\begin{proof}[Proof of \autoref{8}]
The Lagrangian associated with the Hamiltonian $H_a$ by the Legendre transform reads
\begin{align*}
	L(q_1,q_2,\dot{q_1},\dot{q_2})=\frac{1}{2}\dot{q}^TE^{-1}(q)\dot{q}-
	\frac{a_1}{2}(q_1-q_2).
\end{align*}
Note that $E^{-1}(q):=\frac{1}{1-e^{-2|q_1-q_2|}}\begin{bmatrix}
	1 & -e^{-|q_1-q_2|}\\-e^{-|q_1-q_2|} & 1
\end{bmatrix}$.\\
\\
Next, $E^{-1}(q)$ in new coordinates is (see \cite[Theorem 3.1]{CGKM})
\begin{align*}
	A(s_2)&=\frac{1}{1-e^{-4|s_2|}}\begin{bmatrix}
	\frac{\partial q_1}{\partial s_1} & \frac{\partial q_1}{\partial s_2}\\\frac{\partial q_2}{\partial s_1} &  \frac{\partial q_1}{\partial s_2}
	\end{bmatrix}^T\begin{bmatrix}
	1 & -e^{-2|s_2|}\\-e^{-2|s_2|} & 1
	\end{bmatrix}\begin{bmatrix}
	\frac{\partial q_1}{\partial s_1} & \frac{\partial q_1}{\partial s_2}\\\frac{\partial q_2}{\partial s_1} &  \frac{\partial q_1}{\partial s_2}
	\end{bmatrix}\\&=\begin{bmatrix}
	\frac{2}{1+e^{-2|s_2|}} & 0\\0 & \frac{2}{1-e^{-2|s_2|}}
	\end{bmatrix},
\end{align*}
and new form of a Lagrangian is
\begin{align}\label{15}
	L(s_1,s_2,\dot{s_1},\dot{s_2})=\frac{1}{2}\dot{s}^TA(s_2)\dot{s}-a_1s_2.
\end{align}
Next, via the Legendre transform we obtain that a Hamiltonian corresponding to $(\ref{15})$ is
\begin{align*}
	H(\tilde{p_1},\tilde{p_2},s_1,s_2)&=\frac{1}{2}\tilde{p}^TA^{-1}(s_2)\tilde{p}+a_1 s_2,\\&=\frac{1}{2}\begin{pmatrix}
	\tilde{p_1}&\tilde{p_2}
	\end{pmatrix}\begin{bmatrix}
	\frac{1+e^{-2|s_2|}}{2}&0\\0&\frac{1-e^{-2|s_2|}}{2}
	\end{bmatrix}\begin{pmatrix}
	\tilde{p_1}\\\tilde{p_2}
	\end{pmatrix}+a_1 s_2,\\&=\frac{1}{2}\left(\tilde{p_1}^2\frac{1+e^{-2|s_2|}}{2}+\tilde{p_2}^2\frac{1-e^{-2|s_2|}}{2}\right)+a_1 s_2.
\end{align*}
Since $q_1<q_2$, we have $s_2<0$. Then the Hamiltonian equations corresponding to $H(\tilde{p_1},\tilde{p_2},s_1,s_2)$ read
\begin{eqnarray*}
\dot{\tilde{p_1}}&=&-\frac{\partial H}{\partial s_1},\\
\dot{\tilde{p_2}}&=&-\frac{\partial H}{\partial s_2}=\frac{1}{2}\left(\tilde{p_2}^2-\tilde{p_1}^2\right)e^{2s_2}-a_1, \\
\dot{s_1} &=&\frac{\partial H}{\partial\tilde{p_1}}=\frac{\tilde{p_1}}{2}\left(1+e^{2s_2}\right),\\
\dot{s_2}&=&\frac{\partial H}{\partial\tilde{p_2}}=\frac{\tilde{p_2}}{2}\left(1-
e^{2s_2}\right).
\end{eqnarray*}
Since $H$ does not depend on $s_1$, we obtain that $\tilde{p_1}(t)=\tilde{p_1}(0)$. In other words, $\tilde{p_1}$ is the constant of motion.
\end{proof}
We are now in a position to prove controllabilty of a collision of $(\ref{11})-(\ref{14})$. To this end we use the reduced system $(\ref{19})-(\ref{22})$. Notice that to obtain a collision, we need to observe $s_2(t^*)=0$ for $t^*<\infty$. Since $\tilde{p_1}$ is constant, the evolution of $s_2$ is fully covered by $(\ref{20})$ and $(\ref{22})$. Thus all the information concerning collision is contained in $(\ref{20})$ and $(\ref{22})$. Hence, our problem reduces to a 2D problem and it is much easier to handle.
\begin{proposition}\label{23}
Let us assume that $q_1(0)<q_2(0)$. Moreover, consider 
\[
\alpha\geq\frac{1}{2}\tilde{p_1}^2(0)=\frac{1}{2}\left(p_1(0)+p_2(0)\right)^2.
\] 
Then, the collision of a solution to $(\ref{11})-(\ref{14})$ is controllable. In other words, there exists a control $-\alpha\leq a_1(t)\leq \alpha$ and a time $\tau<\infty$, such that a two-peakon satisfying $(\ref{11})-(\ref{14})$, collides.
\end{proposition}
\begin{proof}
We have already noticed that $\tilde{p_1}(t)=\tilde{p_1}(0)$. If we choose $a_1=-\frac{1}{2}\tilde{p_1}^2(0)$ then in view of $(\ref{20})$
\begin{align}\label{16}
	\dot{\tilde{p_2}}=\frac{1}{2}\tilde{p_2}^2(t)e^{2s_2}+\frac{1}{2}\tilde{p_1}^2(0)\left(1-e^{2s_2}\right)>0.
\end{align}
This implies $\tilde{p_2}$ is increasing in time and there exists a time $t^*$ such that $\tilde{p_2}(t^*)=0$. Hence $\tilde{p_2}(t)>\tilde{p_2}(\bar{t})>0$ for $t>\bar{t}>t^*$ and in view of $(\ref{16})$ and $(\ref{22})$, we obtain
\begin{align*}
	\dot{\tilde{p_2}}(t)\geq \frac{1}{2}\tilde{p_2}^2(t)e^{2s_2(t)}\geq \frac{1}{2}\tilde{p_2}^2(t)e^{2s_2(\bar{t})}.
\end{align*}
Integrating the above estimate from $\bar{t}$ to $t$ we get
\begin{align}\label{17}
	\tilde{p_2}(t)\geq \frac{1}{\frac{1}{\tilde{p_2}(\bar{t})}-\frac{1}{2}e^{2s_2(\bar{t})}(t-\bar{t})}.
\end{align}
Next, $(\ref{22})$ and $(\ref{17})$ yield
\begin{align*}
	 \int_{s_2(\bar{t})}^{s_2(t)}\frac{ds_2}{1-e^{2s_2}}=\frac{1}{2}\int_{\bar{t}}^{t}\tilde{p_2}(s)ds\geq\frac{1}{2}\int_{\bar{t}}^{t}\frac{1}{\frac{1}{\tilde{p_2}(\bar{t})}-\frac{1}{2}e^{2s_2(\bar{t})}(s-\bar{t})}ds.
\end{align*}
After integration we get
\begin{align}
\ln\left(e^{-2s_2(t)}-1\right)&\leq \ln\left(e^{-2s_2(\bar{t})}-1\right)+2e^{-2s_2(\bar{t})}\left(\ln\left(\frac{1}{\tilde{p_2}(\bar{t})}-\frac{1}{2}e^{2s_2(\bar{t})}(t-\bar{t})\right)+\ln\tilde{p_2}(\bar{t})\right),\nonumber\\&\leq \ln\left(e^{-2s_2(\bar{t})}-1\right)+\ln\left(1-\frac{1}{2}\tilde{p_2}(\bar{t})e^{2s_2(\bar{t})}(t-\bar{t})\right)^{2e^{-2s_2(\bar{t})}},\nonumber\\ \mbox{i.e.}\ \ \left(e^{-2s_2(t)}-1\right)&\leq \left(e^{-2s_2(\bar{t})}-1\right)\left(1-\frac{1}{2}\tilde{p_2}(\bar{t})e^{2s_2(\bar{t})}(t-\bar{t})\right)^{2e^{-2s_2(\bar{t})}}\label{18}.
\end{align}
Now from inequality $(\ref{18})$, we observe that if $t$ is not later than $\bar{t}+\frac{2e^{-2s_2(\bar{t})}}{\tilde{p_2}(\bar{t})}$ then the right-hand side of $(\ref{18})$ equals $0$. That is equivalent to say $s_2(\tau)=0$ or $q_1(\tau)=q_2(\tau)$ where $\tau\leq \bar{t}+\frac{2e^{-2s_2(\bar{t})}}{\tilde{p_2}(\bar{t})}$. This implies the controllability of a collision of $(\ref{11})-(\ref{14})$.
\end{proof}
Let us notice a useful fact that at a collision time $\tau$, $\tilde{p_2}(\tau)=+\infty$. Indeed, we have the following remark.
\begin{rem}\label{3.2}
Assume $\tau$ is a collision time in $(\ref{41})$. In particular $s_2(\tau)=0$. Then $\tilde{p_2}(\tau)=\infty$.
\end{rem}
\begin{proof}
	Integrating $(\ref{22})$ on interval $[0,\tau]$, we obtain
\begin{align}\label{ln}
	\ln \left(e^{-2s_2(0)}-1\right)-\ln\left(e^{-2s_2(\tau)-1}\right)=\int_{\tilde{p_2}(0)}^{\tilde{p_2}(\tau)}\tilde{p_2} \ dt.
\end{align}
Since $s_2(t)=0$ at collision time $\tau$, $\lim_{t\to\tau}\ln\left(e^{-2s_2(t)-1}\right)=-\infty$. Hence $(\ref{ln})$ yields
\begin{align*}
	\int_{\tilde{p_2}(0)}^{\tilde{p_2}(\tau)}\tilde{p_2} \ dt=\infty,
\end{align*}
and we see that $\tilde{p_2}(\tau)=\infty$. This concludes that $\lim_{t\to\tau}\tilde{p_2}(t)=\infty$.
\end{proof}
\section{Time-Optimal control}
In the previous section, we have shown that a collision of $(\ref{11})-(\ref{14})$ is fully controllable. It is enough to pick up a constant control to this end. This section is devoted to the discussion of a time-optimal control $a_1\in[-\alpha,\alpha]$. We shall show that a constant strategy $a_1=-\alpha$ is a time-optimal control. This way, we finish the proof of Theorem \ref{2.}. 

We have already reduced $(\ref{11})-(\ref{14})$ into the following plane problem
\begin{eqnarray}
\dot{\tilde{p_2}}&=&\frac{1}{2}\left(\tilde{p_2}^2-\tilde{p_1}^2(0)\right)e^{2s_2}-a_1,\label{41} \\
\dot{s_2}&=&\frac{\tilde{p_2}}{2}\left(1-e^{2s_2}\right),\label{41.5}\\ \tilde{p_2}(0)&=&\tilde{p_2}^*, \  s_2(0)=s_2^*,
\label{41.6}
\end{eqnarray}
where $a_1\in[-\alpha,\alpha]$ and $\alpha\geq\frac{1}{2}\tilde{p_1}^2(0)$.\\
\\
Our task is to find a time-optimal control leading to a collision. Although $(\ref{41})-(\ref{41.6})$ is a plane problem, it seems that the approach of Sussmann, Piccoli et al. (see \cite{HS}, \cite{BS} and \cite{BP}) is not straightforward due to the singularity of $(\ref{41})$. Indeed, we carry a solution of $(\ref{41})-(\ref{41.6})$ to a position $\left(\tilde{p_2},s_2\right)=\left(+\infty, 0\right)$. Instead of adjusting the switching function approach of Sussmann, we perform a dynamic programming analysis.

Let us also mention that the result we arrive at is perhaps not surprising. The time-optimal control is such that we pull the first peak up, as much as it's possible, and at the same time, push the second peak down as much, as it gets. So, we try to adjust the situation to the configuration, which allows a collision (\cite{BSS}). However, multipeakon's equations are highly nonlinear and it is initially not clear why the nonlinear effects cannot prevent the mentioned strategy. Our theorem proves that they can't. 

We shall prove our result in several steps. First, we notice that if the time-optimal strategy exists, it satisfies the corresponding Hamilton-Jacobi-Bellman equation. Let us first introduce the value function.   
\begin{definition}\label{4.0}
Let us define a time-minimum function $V:\mathbb{R}\times(-\infty,0]\to[0,\infty)$ by
\begin{align*}
V(\tilde{p_2}^*,s_2^*):=\min\left\{\tau\geq 0 \ | \  s_2(\tau)=0 \ \mbox{and} \ s_2 \ \mbox{is a solution of}\ (\ref{41})-(\ref{41.6})\right\}.
\end{align*}
\end{definition}
\begin{rem}\label{4.}
From \autoref{3.2}, $\tilde{p_2}=\infty$ and $s_2=0$ at the point of collision. Hence we see that $V(\infty,0)=0$.
\end{rem}
Next, following Bellman's strategy, we arrive at the corresponding Hamilton-Jacobi-Bellman boundary problem.
\begin{proposition}
Let us consider time-optimal control problem $(\ref{41})-(\ref{41.6})$. Then, if the time-optimal strategy exists, the value function $V$ satisfies the following problem
\begin{eqnarray}
\frac{\tilde{p_2}}{2}(1-e^{2s_2})\frac{\partial V}{\partial s_2}+\frac{1}{2}\left(\tilde{p_2}^2-\tilde{p_1}^2(0)\right)e^{2s_2}\frac{\partial V}{\partial \tilde{p_2}}-\alpha \left|\frac{\partial V}{\partial \tilde{p_2}}\right|+1=0,\label{42}\\V(\infty,0)=0.\label{v}
\end{eqnarray}
\end{proposition}
\begin{proof}
Let us set cost function as $J(\tilde{p_2},s_2,a_1)=\int_{0}^{\tau} 1 \ ds$, where $\tau$ is a collision time. Since $V$ is independent of time, the Hamilton-Jacobi-Bellman (HJB) equation related to $(\ref{41})-(\ref{41.6})$ [see for instance \cite{JZ} or \cite{AB}] reads as
\begin{align*}
\min_{a_1\in[-\alpha,\alpha]}\left\{\left(\frac{\partial V}{\partial \tilde{p_2}},\frac{\partial V}{\partial s_2}\right)\cdot\left(\frac{1}{2}\left(\tilde{p_2}^2-\tilde{p_1}^2(0)\right)e^{2s_2}-a_1,\frac{\tilde{p_2}}{2}(1-e^{2s_2})\right)+1\right\}=0,
\end{align*}
or
\begin{align*}
\min_{a_1\in[-\alpha,\alpha]}\left\{\frac{\tilde{p_2}}{2}(1-e^{2s_2})\frac{\partial V}{\partial s_2}+\left(\frac{1}{2}\left(\tilde{p_2}^2-\tilde{p_1}^2(0)\right)e^{2s_2}-a_1)\right)\frac{\partial V}{\partial \tilde{p_2}}+1\right\}=0,
\end{align*}	
equivalently (assuming $\frac{\partial V}{\partial \tilde{p_2}}\neq 0$)
\begin{align*}
\frac{\tilde{p_2}}{2}(1-e^{2s_2})\frac{\partial V}{\partial s_2}+\frac{1}{2}\left(\tilde{p_2}^2-\tilde{p_1}^2(0)\right)e^{2s_2}\frac{\partial V}{\partial \tilde{p_2}}-\alpha \left|\frac{\partial V}{\partial \tilde{p_2}}\right|+1=0.
\end{align*}
By \autoref{4.}, $V(\infty,0)=0$.
\end{proof}
A next step in Bellman's approach is to solve $(\ref{42}), (\ref{v})$. Notice that $(\ref{42}), (\ref{v})$ looks quite unpleasant. Due to the nonlinearity, one might consider an application of viscosity solutions theory in that case (see \cite{AB},\cite{LC} or \cite{JZ}). But, the boundary condition (\ref{v}) is highly non-standard. Moreover, we anyway use a trick which allows us to reduce this problem to the linear first order Hamilton-Jacobi equation. A crucial idea is to consider an Ansatz that $V$, solving $(\ref{42}), (\ref{v})$, is decreasing in $\tilde{p_2}$. This way, we reduce the problem to the linear one with a friendly structure. Even the non-standard boundary condition fits to our approach. At the end of the reasoning we check that the solution of the linear problem that we obtained is indeed monotone in $\tilde{p_2}$ and so it satisfies the original nonlinear problem.
\begin{lemma}\label{4.1}
Let us consider a problem $(\ref{42})$, $(\ref{v})$. Then there exists a $\mathcal{C}^1$ solution of $(\ref{42})$ attaining the boundary value $V(\infty,0)=0$. Moreover,
\begin{align*}
V(\tilde{p_2},s_2):=\min\left\{\tau\geq 0 \ | \  \tilde{p_2}(\tau)=\infty, s_2(\tau)=0, \ \mbox{where} \ \tilde{p_2}(t),s_2(t)\ \mbox{solve}\ (\ref{a}),(\ref{b})\right.\\ \left.\mbox{and attaining initial conditions}\  \tilde{p_2}(0)=\tilde{p_2}, s_2(0)=s_2\right\} ,
\end{align*}
and
\begin{eqnarray}
\frac{d\tilde{p_2}}{d\tau}&=&\frac{1}{2}\left(\tilde{p_2}^2-\tilde{p_1}^2(0)\right)e^{2s_2}+\alpha,\label{a} \\ \frac{ds_2}{d\tau}&=&\frac{\tilde{p_2}}{2}(1-e^{2s_2})\label{b}.
\end{eqnarray}
\end{lemma}

\begin{proof}
Let us first consider a linear problem
\begin{eqnarray}
\frac{\tilde{p_2}}{2}(1-e^{2s_2})\frac{\partial \tilde{V}}{\partial s_2}+\frac{1}{2}\left(\tilde{p_2}^2-\tilde{p_1}^2(0)\right)e^{2s_2}\frac{\partial \tilde{V}}{\partial \tilde{p_2}}+\alpha \frac{\partial \tilde{V}}{\partial \tilde{p_2}}+1=0,\label{45}\\ \tilde{V}(\infty,0)=0\label{V}
\end{eqnarray}
The characteristic ODEs associated with $(\ref{45})$, $(\ref{V})$ are $(\ref{a}), (\ref{b})$ and
\begin{align}
\frac{d\tilde{V}}{d\tau}&=-1 ,\ \tilde{V}(\infty,0)=0.\label{c}
\end{align}
Note that, by Proposition \ref{23}, for any $\left(\tilde{p_2}, s_2\right)\in\mathbb{R}\times(-\infty,0]$ there exists a unique regular trajectory of $(\ref{a}), (\ref{b})$ and $(\ref{c})$ that goes through $(\tilde{p_2},s_2)$ . Along the characteristics, we can solve $(\ref{45})$, $(\ref{V})$ and obtain $\tilde{V}(\tilde{p_2},s_2)=\tau$. Hence, we obtain a $\tilde{V}(\tilde{p_2},s_2)$ that solves $(\ref{45})$, $(\ref{V})$.\\
\\
Next, we show that $\tilde{V}$ is a $\mathcal{C}^1$ function. Take $h>0, \tilde{p_2}\in\mathbb{R}, s_2<0$, in view of $(\ref{b})$ we have
\begin{align}
\tilde{V}(\tilde{p_2}, s_2+h)&=\int_{s_2+h}^{0}\frac{2\ ds_2}{\tilde{p_2}\left(1-e^{2s_2}\right)},\label{d_1}\\
\tilde{V}(\tilde{p_2} , s_2)&=\int_{s_2}^{0}\frac{2\ ds_2}{\tilde{p_2}\left(1-e^{2s_2}\right)}\label{d_2}.	
\end{align}
Now using $(\ref{d_1})$ and $(\ref{d_2})$ we obtain
\begin{align*}
	\lim_{h\to 0}\frac{\tilde{V}(\tilde{p_2},s_2+h)-\tilde{V}(\tilde{p_2},s_2)}{h}&=-\lim_{h\to 0}\frac{1}{h}\int_{s_2}^{s_2+h}\frac{2\ ds_2}{\tilde{p_2}\left(1-e^{2s_2}\right)}\\&=-\frac{2}{\tilde{p_2}\left(1-e^{2s_2}\right)}.
\end{align*}
Hence, $\frac{\partial\tilde{ V}}{\partial s_2}$ exists and it is continuous.

Next, we recall that $\alpha\geq\frac{1}{2}\tilde{p_1}^2(0)$.
Again applying similar argument on $\tilde{p_2}$ and using $(\ref{a})$, we obtain
\begin{align}
	\tilde{V}(\tilde{p_2}+h,s_2)&=\int_{\tilde{p_2}+h}^{\infty}\frac{2 \ d\tilde{p_2}}{\tilde{p_2}^2e^{2s_2}+\tilde{p_1}^2(0)\left(1-e^{2s_2}\right)+\left(2\alpha-\tilde{p_1}^2(0)\right)},\label{46}\\ \tilde{V}(\tilde{p_2},s_2)&=\int_{\tilde{p_2}}^{\infty}\frac{2 \ d\tilde{p_2}}{\tilde{p_2}^2e^{2s_2}+\tilde{p_1}^2(0)\left(1-e^{2s_2}\right)+\left(2\alpha-\tilde{p_1}^2(0)\right)}.\label{47}
\end{align}
In particular, $(\ref{46})$ and $(\ref{47})$ yield
\begin{align}\label{V.}
	\lim_{h\to 0}\frac{\tilde{V}(\tilde{p_2}+h,s_2)-\tilde{V}(\tilde{p_2},s_2)}{h}&=-\lim_{h\to 0}\frac{1}{h}\int_{\tilde{p_2}}^{\tilde{p_2}+h}\frac{2\ d\tilde{p_2}}{\tilde{p_2}^2e^{2s_2}+\tilde{p_1}^2(0)\left(1-e^{2s_2}\right)+\left(2\alpha-\tilde{p_1}^2(0)\right)}<0.
\end{align}
Hence $\frac{\partial \tilde{V}}{\partial \tilde{p_2}}$ exists and it is continuous. So, we conclude that $V(\tilde{p_2},s_2)$ is a $\mathcal{C}^1$ function. Moreover, in view of $(\ref{V.})$, we also have
\begin{align}\label{t}
	\frac{\partial \tilde{V}}{\partial \tilde{p_2}}<0.
\end{align}
Hence, in view of $(\ref{t})$, we notice that $\tilde{V}$ solves not only $(\ref{45}),(\ref{V})$ but also $(\ref{42}), (\ref{v})$. The latter is the Hamilton-Jacobi-Bellman equation related to $(\ref{41})-(\ref{41.6})$.
\end{proof}
In \autoref{4.1} we have shown that $V$ is a solution of the Hamilton-Jacobi-Bellman equation $(\ref{42}), (\ref{v})$, related to the time-optimal control problem $(\ref{41})-(\ref{41.6})$. Hence, a standard dynamic programming argument would give us an optimal control as an argmin of the Hamiltonian. But, our problem $(\ref{41})-(\ref{41.6})$ is not a regular one, as $\tilde{p_2}$ blows up in finite time, the assumptions under which one usually formulates the Bellman principle are not met. However, $V$ is a $\mathcal{C}^1$ function, we can still use the classical method, we provide the proof for reader's convenience.
\begin{lemma}\label{4.3}
Let us consider time-optimal problem $(\ref{41})-(\ref{41.6})$. A constant strategy $a_1=-\alpha$ is optimal.
\end{lemma}
\begin{proof}
Hamilton-Jacobi-Bellman equation for control problem $(\ref{41})-(\ref{41.6})$ reads
\begin{align*}
\min_{a_1\in[-\alpha,\alpha]}\left\{\frac{\tilde{p_2}}{2}(1-e^{2s_2})\frac{\partial V}{\partial s_2}+\left(\frac{1}{2}\left(\tilde{p_2}^2-\tilde{p_1}^2(0)\right)e^{2s_2}-a_1)\right)\frac{\partial V}{\partial \tilde{p_2}}+1\right\}=0,
\end{align*}
which is equivalent to
\begin{align}\label{e}
\frac{\tilde{p_2}}{2}(1-e^{2s_2})\frac{\partial V}{\partial
	 s_2}+\frac{1}{2}(\tilde{p_2}^2-\tilde{p_1}^2(0))e^{2s_2}\frac{\partial V}{\partial \tilde{p_2}}+\min_{a_1\in[-\alpha,\alpha]}\left[-a_1\frac{\partial V}{\partial \tilde{p_2}}\right]+1=0.
\end{align}
By $(\ref{t})$ we have $\frac{\partial V}{\partial \tilde{p_2}}<0$. Hence argmin of the left-hand side of $(\ref{e})$ is attained for $a_1=-\alpha$.\\
\\
Let us denote by $\alpha^*=-\alpha$. Now, we need to show that $\alpha^*$ is an optimal control. Hamilton-Jacobi-Bellman equation related to $(\ref{41})-(\ref{41.6})$ reads
\begin{align*}
	\nabla_{(\tilde{p_2},s_2)}V(\tilde{p_2},s_2)\cdot f(\tilde{p_2},s_2,\alpha^*)+1=0.
\end{align*}
Solution $V$ is regular, thus we have
\begin{align*}
	J(\tilde{p_2},s_2,\alpha^*)&=\int_{0}^{\tau} 1\ ds=-\int_{0}^{\tau} \nabla_{(\tilde{p_2},s_2)}V(\tilde{p_2},s_2)\cdot f(\tilde{p_2},s_2,\alpha^*)\ ds\\&=-\int_{0}^{\tau} \nabla_{(\tilde{p_2},s_2)}V(\tilde{p_2},s_2)\cdot \left(\dot{\tilde{p_1}}, \dot{s_2}\right) \ ds\\&=-\int_{0}^{\tau}\frac{d}{ds}\ V(\tilde{p_2},s_2)\ ds\\&=V(\tilde{p_2}(0),s_2(0))-V(\tilde{p_2}(\tau),s_2(\tau))=\min_{a_1}J(\tilde{p_2},s_2,a_1).
\end{align*}
This concludes that $a_1=-\alpha$ is indeed an optimal control.
\end{proof}
Next, we conlude that our main Theorem \ref{2.} is proven. Proposition \ref{23} gives us controllability. Lemma \ref{4.3}, due to the equivalence of (\ref{11})-(\ref{14}) and (\ref{41})-(\ref{41.6}), finishes the proof of Theorem \ref{2.}. 

\section{Special version of a maximum principle for more general systems}

Our proof of the time-optimality of Theorem \ref{2.} can be extended to more general problems. In this section we state such extensions and sketch their proofs. 

One may notice that Hamiltonian related to $(\ref{41})-(\ref{41.6})$ is 
\[
H(\tilde{p_2},s_2,a_1)=\frac{1}{4}\left(\tilde{p_2}^2-\tilde{p_1}^2(0)\right)\left(1-e^{2s_2}\right)+a_1s_2
\]
and $H$ also attains maximum at $a_1=-\alpha$. Moreover, $a_1=-\alpha$ is an optimal control for $(\ref{41})-(\ref{41.6})$. This suggests a special version of maximum principle for a set of time-optimal control problems with a Hamiltonian structure on a plane. Indeed, the following proposition holds true.
\begin{proposition}\label{4.5}
Let us consider the following $2D$ time-optimal control problem 
\begin{eqnarray}
\dot{p}&=&-\frac{\partial F(p,q)}{\partial q}-a_1,\label{p}\\
\dot{q}&=& \frac{\partial F(p,q)}{\partial p}\label{s},
\end{eqnarray}
where $p:[0,T]\to\mathbb{R}$, $q:[0,T]\to(-\infty,0]$, $a_1:[0,T]\to[-\alpha, \alpha]$ for $\alpha>0$ and $F(p,q)$ is a $\mathcal{C}^2$ function. Further, assume that $\tilde{p}, \tilde{q}$ solve the following equations
\begin{eqnarray}
\dot{\tilde{p}}&=&-\frac{\partial F(\tilde{p},\tilde{q})}{\partial \tilde{q}}+\alpha,\label{p.}\\
\dot{\tilde{q}}&=&\frac{\partial F(\tilde{p},\tilde{q})}{\partial \tilde{p}}\label{s.},\\\tilde{p}(0)&=&p(0), \ \tilde{q}(0)=q(0),\nonumber
\end{eqnarray}
and satisfy the following conditions
\begin{itemize}
\item[i)] for any $(\tilde{p}, \tilde{q})\in \mathbb{R}\times(-\infty,0]$, there exists a unique trajectory of (\ref{p.}) and (\ref{s.}) that goes through $(\tilde{p},\tilde{q})$ and satisfies $\left(\tilde{p}(T), \tilde{q}(T)\right)=(+\infty, 0)$ for $T<\infty$,
\item[ii)] If $(\tilde{p}',\tilde{q})$ and $(\tilde{p}'',\tilde{q})$ are two pairs of initial conditions of (\ref{p.}) and (\ref{s.}), satisfying $\tilde{p}'>\tilde{p}''$, then time needed to reach $(+\infty,0)$ from the first pair of initial conditions is smaller than from the latter ones.
\end{itemize}
Then a constant strategy $a_1=-\alpha$ is optimal for $(\ref{p})-(\ref{s})$.
\end{proposition}
\begin{proof}
The proof follows the lines of the proof of Lemma \ref{4.1} and Lemma \ref{4.3}, so we only sketch it. 
First, we define time-minimum function
\begin{align*}
	W(p,q)=\left\{ T>0 \ | \  \ p(t),q(t) \ \mbox{solve}\ (\ref{p}),(\ref{s}) \ \mbox{with} \ p(T)=\infty, q(T)=0 \ \mbox{and}\ p(0)=p, q(0)=q\ \right\}.
\end{align*}
Observe that $W(p(T),q(T))=W(\infty,0)=0$. Dynamic programming principle shows that provided optimal strategy exists, $W$ satisfies the Hamilton-Jacobi-Bellman problem 
\begin{eqnarray}
	\frac{\partial F}{\partial p}\ \frac{\partial W}{\partial q}-\frac{\partial F}{\partial q}\ \frac{\partial W}{\partial p}-\alpha\left|\frac{\partial W}{\partial p}\right|+1=0,\label{p_1}\\
	W(\infty,0)=0\label{s_1}.
\end{eqnarray}
Next, we define $\tilde{W}$ with the use of $\tilde{p}, \tilde{q}$, solutions to (\ref{p.})-(\ref{s.}), in an analogous way as $\tilde{V}$ in Lemma \ref{4.1}.
 
Next, owing to $i)$, again following the lines in Lemma \ref{4.1},  we show that $\tilde{W}$ solves a linear PDE
\begin{eqnarray}
\frac{\partial F}{\partial p}\ \frac{\partial \tilde{W}}{\partial q}-\frac{\partial F}{\partial q}\ \frac{\partial \tilde{W}}{\partial p}+\alpha\frac{\partial \tilde{W}}{\partial p}+1=0,\label{p_i}\\
\tilde{W}(\infty,0)=0\label{s_ii}.
\end{eqnarray}
Condition $ii)$, translated in terms of $\tilde{W}$, gives
\begin{align}\label{w}
	\frac{\partial \tilde{W}}{\partial p}<0.
\end{align}
The latter allows us to say that $\tilde{W}$ also solves $(\ref{p_1}),(\ref{s_1})$.

Since $W$ is a solution of the Hamilton-Jacobi-Bellman equation of $(\ref{p_1}),(\ref{s_1})$, related to the time-optimal control problem $(\ref{p}),(\ref{s})$, we are in a position to point a time-optimal control. Hamilton-Jacobi-Bellman equation related to $(\ref{p}),(\ref{s})$ reads as
\begin{align*}
	\frac{\partial F}{\partial p}\ \frac{\partial W}{\partial q}-\frac{\partial F}{\partial q}\ \frac{\partial W}{\partial p}+\min_{a_1\in[-\alpha,\alpha]}\left[-a_1\frac{\partial W}{\partial p}\right]+1=0.
\end{align*}
Hence, it follows that $a_1=-\alpha$ is an optimal control using similar exposition as in Lemma \ref{4.3}.
\end{proof}
We state below the fact that, according to Proposition \ref{23} and (\ref{t}), conditions (i) and (ii) of Proposition \ref{4.5} are satisfied by our original problem.  
\begin{rem}
	Our original time-optimal control problem $(\ref{41})-(\ref{41.6})$ satisfies conditions $i), ii)$ of Proposition \ref{4.5}. 
\end{rem}

Next, we state a multidimensional, even more general, version of time-optimal maximum principle which 
states the no-switching bang-bang controls as time-optimal ones in some cases, hopefully, of interest.  
\begin{proposition}
Let $a_i:[0,T]\to[-\alpha_i,\alpha_i]$ for $i=1,\ldots,n$, $\alpha_i>0$. Let us consider a time-optimal control problem as follows
\begin{eqnarray}
	\dot{p_i}&=&-\frac{\partial G(p,q)}{\partial q_i}-a_i,\label{p_2}\\
	\dot{q_i}&=&\frac{\partial G(p,q)}{\partial p_i}\label{s_2},
\end{eqnarray}
where $p_i:[0,T]\to\mathbb{R}$, $q_i:[0,T]\to(-\infty,0]$, $p=\left(p_1,\ldots,p_n\right), q=\left(q_1,\ldots,q_n\right)$ and $G(p,q)$ is a $\mathcal{C}^2$ function. Further, assume that $\tilde{p_i},\tilde{q_i}$ solve the following equations
\begin{eqnarray}
\dot{\tilde{p_i}}&=&-\frac{\partial G(\tilde{p},\tilde{q})}{\partial {q_i}}+\alpha_i,\label{p_3}\\
\dot{\tilde{q_i}}&=&\frac{\partial G(\tilde{p},\tilde{q})}{\partial \tilde{p_i}}\label{s_3},\\\tilde{p_i}(0)&=&p_i(0), \ \tilde{q_i}(0)=q_i(0),\nonumber
\end{eqnarray}
and moreover $\tilde{p}=\left(\tilde{p_1},\ldots,\tilde{p_n}\right),\tilde{q}=\left(\tilde{q_1},\ldots,\tilde{q_n}\right)$ satisfy the following conditions
\begin{itemize}
	\item[i)] for any $(\tilde{p}, \tilde{q})\in \mathbb{R}^n\times(-\infty,0]^n$, there exists a unique trajectory of $(\ref{p_3})$ and $(\ref{s_3})$ that goes through $(\tilde{p},\tilde{q})$ and hits the the terminal point $\left(\tilde{p}(T), \tilde{q}(T)\right)=(+\infty,...,+\infty,0,...,0)$,
	\item[ii)] Let $(p',q)$ and $(p'',q)$ be two distinct initial pairs such that for any $i\in \{1,...,n\}$, $p'_j=p''_j$, $j\neq i$, $\tilde{p}'_i>\tilde{p}''_i$. Then the time needed to reach $(+\infty,...,0,...,0))$ from $(p',q)$ is smaller than the time needed when starting from $(p'',q)$.
\end{itemize}
Then a constant strategy $a_i=-\alpha_i$ is a time-optimal one for $(\ref{p_2})-(\ref{s_2})$.
\end{proposition}
\begin{proof}
The proof is analogous to the proof of Proposition \ref{4.5}. Hence, we shall not discuss it in full details.
To shorten the notation let us recall the definition of the Poisson bracket of two regular functions $F(p,q),G(p,q)$, 
\[
\{F,G\}:=\sum_{i=1}^n\frac{\partial F}{\partial p_i}\frac{\partial G}{\partial q_i}-\sum_{i=1}^n\frac{\partial F}{\partial q_i}\frac{\partial G}{\partial p_i}.
\]
Next, following the previous strategy, we define a time-minimum function
\begin{align*}
Y(p,q)=\left\{ T>0 \ | \  \ p,q \ \mbox{solve}\ (\ref{p_2}),(\ref{s_2}) \ \mbox{with} \ p_i(T)=\infty, q_i(T)=0, i=1,..,n \ \mbox{and}\ p(0)=p,q(0)=q\ \right\}.
\end{align*}
Observe that $Y(p(T),q(T))=0$. Now, Bellman's principle allows to say that, if the time-optimal strategy exists, then the following Hamilton-Jacobi-Bellman equation is satisfied
\begin{eqnarray}
\{G,Y\}-\sum_{i=1}^n\alpha_i\left|\frac{\partial Y}{\partial p_i}\right|+1=0,\label{p_y}\\
Y(\infty,...,0,...,0)=0\label{s_y}.
\end{eqnarray}
Again, exactly the same way as in the proof of Proposition \ref{4.5}, using $i)$ and method of characteristics, we show that $\tilde{Y}$, defined analogously as $\tilde{W}$ in the proof of Proposition \ref{4.5}, solves the corresponding linear PDE
\begin{eqnarray}
\{G,\tilde{Y}\}+\sum_{i=1}^n\alpha_i\frac{\partial \tilde{Y}}{\partial p_i}+1=0,\label{p_yi}\\
\tilde{Y}(\infty,...,0,...,0)=0\label{s_yii}.
\end{eqnarray}
Next, condition $ii)$ translated in terms of $\tilde{Y}$, tells us that solution to $(\ref{p_2})-(\ref{s_2})$ satisfies
\begin{align}\label{y}
\frac{\partial \tilde{Y}}{\partial p_i}<0.
\end{align}
Hence in view of $(\ref{y})$, we notice that $\tilde{Y}$ also solves $(\ref{p_y}),(\ref{s_y})$.\\
\\
Since $Y$ is a solution of the Hamilton-Jacobi-Bellman equation $(\ref{p_y}),(\ref{s_y})$, related to the time-optimal control problem $(\ref{p_2}),(\ref{s_2})$, we are in a position to find a non-switching optimal control. 
Indeed, it follows that $a_i=-\alpha_i, i=1,...,n$, is a time-optimal control similarly as in the proof of Proposition \ref{4.5}.
\end{proof}

\hspace{0.7cm}
\textbf{Acknowledgement.} T.C.  and B.D. were both partially supported by the National Science Centre grant SONATA BIS 7 number UMO-2017/26/E/ST1/00989.

\end{titlepage}
\end{document}